\documentclass[11pt, a4paper]{article}
\usepackage{microtype}
\usepackage{amsmath}
\usepackage{amssymb, units, amsthm}
\usepackage[english]{babel}
\usepackage[cp1250]{inputenc}
\usepackage{color}
\usepackage{pgf}

\usepackage[numbers,sort&compress]{natbib}
\allowdisplaybreaks %allows pagebreaks in equations

\DeclareMathOperator{\dir}{div}

\DeclareMathOperator{\spn}{span}

\newcommand{\dfe}{\mathrel={\mathop:} \,}
\newcommand{\ddf}{\mathrel{\mathop:}=}

\def \0{\boldsymbol{0}}

\def \e{\varepsilon}

\def \hex{H_{\text{ext}}}
\def \ii{\int_{\Omega}}

\def \N{\mathbb{N}}
\def \n{\boldsymbol{n}}
\def \nn{\nabla}

\def \om{\Omega}

\def \R{\mathbb{R}}

\newtheoremstyle{vety}	%  name
  {5mm}			%  Space above
  {5mm}			%  Space below
  {\itshape}		%  Body font
  {}			%  Indent amount (empty = no indent, \parindent = para indent)
  {}			%  Thm head font
  {\bf}		%  Punctuation after thm head
  {.5em}		%  Space after thm head: " " = normal interword space; (\newline = linebreak)
  {\thmname{\textbf{#1}}\thmnumber{\textbf{ #2}}\thmnote{\textnormal{ (#3)}}}%         Thm head spec (can be left empty, meaning `normal')

\theoremstyle{vety}
\newtheorem{thm}{Theorem}%[section]
\newtheorem{lem}[thm]{Lemma}

\newtheoremstyle{definice}	%  name
  {5mm}				%  Space above
  {5mm}				%  Space below
  {\normalfont}			%  Body font
  {}				%  Indent amount (empty = no indent, \parindent = para indent)
  {}				%  Thm head font
  {\bf}			%  Punctuation after thm head
  {.5em}			%  Space after thm head: " " = normal interword space; (\newline = linebreak)
  {\thmname{\textbf{#1}}\thmnumber{\textbf{ #2}}\thmnote{\textnormal{ (#3)}}}%         Thm head spec (can be left empty, meaning `normal')

\makeatletter
\newcommand\ackname{Acknowledgements}
\if@titlepage
  \newenvironment{acknowledgements}{%
      \titlepage
      \null\vfil
      \@beginparpenalty\@lowpenalty
      \begin{center}%
        \bfseries \ackname
        \@endparpenalty\@M
      \end{center}}%
     {\par\vfil\null\endtitlepage}
\else
  
\fi
\makeatother

\makeatletter
\newcommand\keyname{Keywords}
\if@titlepage
  \newenvironment{keywords}{%
      \titlepage
      \null\vfil
      \@beginparpenalty\@lowpenalty
      \begin{center}%
        \bfseries \keyname
        \@endparpenalty\@M
      \end{center}}%
     {\par\vfil\null\endtitlepage}
\else
  \newenvironment{keywords}{%
      \if@twocolumn
        \section*{\abstractname}%
      \else
        \small
        \begin{center}%
          {\bfseries \keyname\vspace{-.5em}\vspace{\z@}}%
        \end{center}%
        \quotation
      \fi}
      {\if@twocolumn\else\endquotation\fi}
\fi
\makeatother

\theoremstyle{definice}
\newtheorem{df}[thm]{Definition}
\newtheorem{rem}[thm]{Remark}

\begin{document}
\title{Uniqueness of solutions for a mathematical model for magneto-viscoelastic flows}

\author{
A.\ Schl\"omerkemper\footnote{Corresponding author: \tt{anja.schloemerkemper@mathematik.uni-wuerzburg.de}}, 
J.\ \v Zabensk\'y\footnote{\tt{josef.zabensky@mathematik.uni-wuerzburg.de}} 
\\ \\ 
\multicolumn{1}{p{.8\textwidth}}
{\centering University of W\"urzburg, Institute of Mathematics, %Chair of Mathematics XI, 
Emil-Fischer-Stra\ss e 40, 97074 W\"urzburg, Germany}
}

\date{\today}

\maketitle

\begin{abstract}
\noindent We investigate uniqueness of weak solutions for a system of partial differential equations capturing behavior
of magnetoelastic materials. This system couples the Navier-Stokes equations with evolutionary equations for the deformation gradient and for the magnetization obtained from a special case of the micromagnetic energy.
It turns out that the conditions on uniqueness coincide with those for the well-known Navier-Stokes equations in bounded domains: weak solutions are unique in two spatial dimensions, and weak solutions satisfying the Prodi-Serrin conditions are unique among all weak solutions in three dimensions. That is, we obtain the so-called weak-strong uniqueness result in three spatial dimensions.

\end{abstract}

\begin{keywords}
 \noindent Uniqueness of solutions, weak solution, weak-strong uniqueness, magnetoelastic materials, magnetoelasticity, Prodi-Serrin conditions
\end{keywords}

%\begin{acknowledgements}
% \noindent will be
%\end{acknowledgements}

\section{Introduction}

Magnetoelastic materials have various applications as fluids and solids in sensors and actuators due to the coupling of magnetic and elastic effects. In particular, such materials respond mechanically to applied magnetic fields and they change their magnetic properties in response to mechanical stresses. The aim of this work is to study uniqueness of weak solutions of a corresponding evolutionary system of partial differential equations in the incompressible case.

Let $\om \subset \R^d$ be a bounded regular domain for $d=2$ or $3$, let $T>0$ be a given time
and denote $Q_T \ddf (0,T) \times \om$ for brevity. Let $v: Q_T \to \R^d$ be the velocity field, $F: Q_T \to \R^{d\times d}$ the deformation gradient, and $M: Q_T\to \R^3$ the magnetization vector, all described in Eulerian coordinates.  The magnetoelastic material is exposed to an external magnetic field $H_{\text{ext}}: Q_T\to\R^3$. We consider a system in which the behavior of the material is governed by the following set of balance equations
\begin{align} \label{prob1} 
v_t + (v \cdot \nabla) v + \nn p + \dir (\nn^\text{T} \! M \, \nn M ) - \dir (FF^\text{T}) & = \nu \Delta v + \nabla^\text{T} H_{\text{ext}} M, \\
\dir v &= 0, \phantom{\frac{1}{\mu^2}} \label{prob2}\\
F_t + (v \cdot \nabla)F - \nn v F & = \kappa \Delta F, \label{prob3}\\
M_t + (v \cdot \nabla)M & = \Delta M + H_{\text{ext}} - \frac{1}{\mu^2}(|M|^2 - 1)M, \label{prob4}
\end{align}
holding in $Q_T$, with the boundary conditions $v=0$, $F=0$ and $\frac{\partial M}{\partial n}=0$ on $(0,T) \times \partial \om$, where $n$ stands for the unit outer normal to $\om$. We complete the system with appropriate initial data specified below.

Existence of weak solutions to the above system 
was derived in \cite{ForsterDiss} for slightly different boundary conditions and $\hex = 0$, cf.\ also \cite{PAMM2015}. For the convenience of the reader we will present the main steps of the existence proof for system \eqref{prob1}--\eqref{prob4} in Appendix~\ref{sec:exist}. The main focus of this article is on uniqueness of weak solutions: In the two-dimensional case ($d=2$), weak solutions turn out to be unique; see Theorem~\ref{thm1}. The proof relies on the fact that the weak solutions may be used as test functions in the weak formulation of \eqref{prob1}--\eqref{prob4} in two spatial dimensions (see discussion under~\eqref{enineq} for a precise explanation). In three dimensions ($d=3$), this is no longer the case, which is why the proof cannot be transferred. Instead we prove a kind of
weak-strong uniqueness, see Theorem~\ref{unik}. A strong solution of the system is a weak solution that satisfies certain additional
regularity conditions. Any such strong solution turns out to be unique among all weak solutions, the property being referred to as weak-strong uniqueness.

A distinguishing feature of the system under consideration is interconnection  of viscoelasticity with magnetism, which results in the incompressible Navier-Stokes equations coupled with balance equations for the deformation gradient $F$ and the magnetization $M$. The resulting nonlinear terms present challenges in pursuit of existence and uniqueness of weak solutions to the system. Nevertheless, it turns out that uniqueness for the coupled system can still be shown by means of the energy method, see Sections~\ref{sec:uni2d} and~\ref{sec:uni3d}.

From the modeling point of view, a difficulty related to the coupled magnetic and elastic systems is that magnetism is phrased in Eulerian coordinates (current configuration) and elasticity in Lagrangian ones (reference configuration, also referred to as initial or undeformed configuration). There are various approaches in the literature (see e.g.~\cite{DeSimoneDolzmann1998,DeSimoneJames2002, JamesKinderlehrer1993}) for variational approaches in the static case. Rate-independent evolution models are studied in \cite{KruzikStefanelliZeman2015}. There are also results available for models which couple the Landau-Lifshitz-Gilbert equation with elasticity in the small strain setting; see \cite{CarbouEfendievFabrie2011,ChipotShafrir_etal2009}. Most closely related to our system \eqref{prob1}--\eqref{prob4} is the system studied in \cite{BenesovaForster_etal2016}, cf.\ also \cite{ForsterDiss}. In their model, \eqref{prob4} is replaced by the Landau-Lifshitz-Gilbert equation, which is more involved than the gradient flow~\eqref{prob4} due to further nonlinearities. Although the existence of weak solutions has been proved, the question of their uniqueness will be investigated in our
future work.

Following \cite{LiuWalkington2001}, we transform all physical quantities to the Eulerian system, which has the advantage that we do not have to deal with invertibility issues of the corresponding deformation mapping later and that it will be easier to extend our model to composite materials. The derivation of \eqref{prob1}--\eqref{prob4} is obtained from a variational approach outlined in Appendix~\ref{sec:deriv}. The system includes micromagnetism, which is a variational theory that allows the magnetization to form microstructures; see the reviews \cite{DKMO,Garcia-Cervera2007,KruzikProhl2006}. 

Note that if $M\equiv 0$, the system \eqref{prob1}--\eqref{prob4} reduces to a model for incompressible viscoelastic flows, cf.\ \cite{HuWu2015,LinLiuZhang2005,LiuWalkington2001}. The idea of regularizing the evolution equation for the deformation gradient~\eqref{prob3} goes back to \cite{LinLiu1995}. For an overview of analytical results for viscoelastic flows we refer to the introduction of~\cite{HuWu2015}, where the system is studied without the regularization, yet under the assumption that the initial data are close to equilibrium.

The situation is similar in spirit to that for the Navier-Stokes equations, where one has uniqueness of weak solutions in two dimensions and weak-strong uniqueness in three spatial dimensions. The regularity of the strong solution is prescribed by the Prodi-Serrin conditions, cf., e.g., \cite{MR0126088,MR0316915} or~\cite[p.~298]{temam1977navier}. That is, to show that a given weak solution is unique, it \emph{suffices} to show e.g.\ that it satisfies the Prodi-Serrin conditions.  As is well-known, uniqueness depends on the geometry of the spatial domain; there are examples of non-uniqueness in certain domains; see \cite{Heywood}. 

Similar problems have been discussed in related but different systems of partial differential equations in recent years. For instance, weak-strong uniqueness is shown for incompressible viscoelastic flows \cite{HuWu2015}, for the flow of nematic liquid crystals \cite{PaicuZarnescu2011}, in 3D incompressible Navier-Stokes equations with damping~\cite{Zhou2012}, and compressible, viscous, and heat conducting fluids 
\cite{FeireislJinNovotny2012}. In the context of measure-valued solutions, weak-strong uniqueness is studied for instance for polyconvex elastodynamics~\cite{DemouliniStuartTzavaras2012}, and for compressible fluid models in~\cite{GwiazdaSwierczewska-GwiazdaWiedemann2015}. Moreover, uniqueness results are also known for the magneto-hydrodynamic equations, cf.~\cite{ShiZhang2016}.

To avoid unnecessarily cluttered presentation, in the entire paper we will assume that the external forces are zero, i.e.~$\hex = 0$. Had this quantity been present, we would only need to change inequality~\eqref{enineq} below in the corresponding manner and then, in the proof itself, we would treat a few more integral terms in the very same way as we do already. The proof would have been slightly longer, but it would offer no new ideas.

The outline of the remaining part of the paper is as follows: In the next, second section, we state the main results on existence and uniqueness of weak solutions. In the third section we prove uniqueness in two dimensions and in the fourth one we show weak-strong uniqueness in three dimensions.  In Appendix~\ref{sec:deriv} we comment on the derivation of the model \eqref{prob1}--\eqref{prob4}. In Appendix~\ref{sec:exist} we then provide a skeleton proof of the existence of weak solutions.

%---%

\section{The result}

Before we formulate the announced theorem on uniqueness, let us first acquaint the reader with our notation. It is mostly standard or self-explanatory; we only recall 
\begin{align*}
W_{0,\dir}^{1,2}(\om) &\ddf \bigl\{ \varphi \in W^{1,2}(\om;\R^d)\mid \dir \varphi = 0 \text{ in $\om$}, \varphi = 0 \text{ on $\partial \om$ in the sense of traces} \bigr\},\\
W_n^{2,2}(\om)&\ddf \bigl\{ \varphi \in W^{2,2}(\om;\R^3)\mid \frac{\partial \varphi}{\partial n}=0 \text{ on $\partial \om$ in the sense of traces} \bigr\}.
\end{align*}
No explicit distinction between spaces of scalar-, vector-, or tensor-valued functions will be made, for confusion should never occur. Generic constants used in estimates are denoted by $C$ and their value can vary between the lines.

Next we define the notion of a weak solution to the investigated system:
\begin{df} \label{defwsol}
Let $T>0$, $v_0\in L^2_{\dir}(\Omega;\R^d)$, $F_0\in L^2(\Omega;\R^{d\times d})$ and $M_0\in W^{1,2}(\Omega;\R^3)$. Denote $r\ddf 2$ if $d=2$ and $r \ddf 4/3$ if $d=3$. We call $(v,F,M): Q_T \to \R^d \times \R^{d\times d} \times \R^3$ a weak solution to \eqref{prob1}--\eqref{prob4}, if
\begin{align}
v & \in L^{\infty}(0,T;L^2(\om)) \cap L^2(0,T;W_{0,\dir}^{1,2}(\om)), \quad &v_t & \in L^r(0,T;(W^{1,2}_{0,\dir}(\om))^*),\nonumber \\
F & \in L^{\infty}(0,T;L^2(\om)) \cap L^2(0,T;W_0^{1,2}(\om)), \quad &F_t & \in L^r(0,T;(W^{1,2}_0(\om))^*),\label{demo} \\ \nonumber
M & \in L^{\infty}(0,T;W^{1,2}(\om)) \cap L^2(0,T;W_n^{2,2}(\om)), \quad & M_t & \in L^r(0,T;L^2(\om)),
\end{align}
equations \eqref{prob1}, \eqref{prob3} and \eqref{prob4} are satisfied in the sense of distributions on $Q_T$ (\eqref{prob1} only for solenoidal test functions, so that we need not have to bother with the pressure $p$, whose treatment we neglect completely) and the inequality 
\begin{align}\nonumber
 \| v(t) \|_{L^2(\om)}^2 & +\| F(t) \|_{L^2(\om)}^2 +\| M(t) \|_{L^2(\om)}^2 +\| \nn M(t) \|_{L^2(\om)}^2 \\
 \label{enineq} & + 2 \int_{Q_t} \big( \nu | \nn v |^2 + \kappa | \nn F |^2+ | \nn M |^2 + | \Delta M |^2  + \frac{1}{\mu^2}(|M|^2-1)M\cdot (M- \Delta M) \big) \, dx\,ds \\ 
 \nonumber & \qquad \qquad  \le \| v_0\|_{L^2(\om)}^2 + \|F_0\|_{L^2(\om)}^2 + \|M_0\|_{L^2(\om)}^2 + \|\nn M_0\|_{L^2(\om)}^2   
\end{align}
holds for a.e.\ $t \in (0,T)$. 

The boundary conditions are assumed to be attained in the form
\begin{align} \label{icond}
\lim_{t \to 0_+} \big( \|v(t)-v_0 \|_{L^2(\om)} +\|F(t)-F_0 \|_{L^2(\om)} +\|M(t)-M_0 \|_{W^{1,2}(\om)} \big)=0 .
\end{align}
\end{df}
Instead of~\eqref{enineq}, we could alternatively require an inequality in the vein of~\cite[(3.111)]{ForsterDiss}. The latter version, although a true \emph{energy inequality} (unlike our case), would on the one hand make the definition of weak solutions much closer to the well-known Leray-Hopf solutions to the Navier-Stokes equations, yet on the other hand, it would be much less practical in our proof. For this reason we chose the \emph{not-quite-energy} inequality~\eqref{enineq} in our definition. By the way, this inequality is unnecessary to impose explicitly when $d=2$, since then regularity~\eqref{demo} implies that we may test the equation for linear momentum~\eqref{prob1} by $v$, equation~\eqref{prob3} by $F$ and equation~\eqref{prob4} by $M-\Delta M$, yielding in the end even equality in~ \eqref{enineq}. 

Validity of \eqref{enineq} is no longer obvious in three dimensions: \eqref{prob1}--\eqref{prob4} and \eqref{demo} then dictate that we would need also $v \in L^4(0,T;L^4(\om;\R^d))$, $F \in L^4(0,T;L^4(\om;\R^{d\times d}))$ and $\nn M \in L^4(0,T; L^4(\om;\R^{3\times d}))$ to use $v$, $F$ and $\Delta M$ as the respective test functions, which is more than the definition assures (see Lemma~\ref{unreal}). Actually, inequality~\eqref{enineq} in three dimensions is a consequence of its being satisfied as equality by approximate regular solutions and weak/weak$^*$ lower semicontinuity of the norm. As a sidenote, the only difference between~\eqref{enineq} and the said true energy inequality, lies in testing equation~\eqref{prob4} by its entire right-hand side in lieu of simply $M-\Delta M$.

Having got acquainted with the concept of weak solutions, we present the theorem on existence of weak solutions, whose proof is sketched in Appendix~\ref{sec:exist}.

\begin{thm}\label{thm:exist} Let $d=2,3$ and $\Omega\subset \R^d$ be a $C^\infty$ domain. Then the system \eqref{prob1}--\eqref{prob4} possesses a weak solution $(v, F, M)$ in the sense of Definition~\ref{defwsol}.
\end{thm}

The main purpose of this paper is to prove the following two statements on uniqueness of the weak solutions found above, which will be proved in the subsequent sections.

\begin{thm} \label{thm1}
When $d=2$, weak solutions are unique.
\end{thm}
In three dimensions, we have a weak-strong type of uniqueness in the vein of Navier-Stokes' equations, dictated by the Prodi-Serrin conditions: 
\begin{thm} \label{unik}
Let $d=3$, $(v^1,F^1,M^1)$ be a weak solution and  $(v^2,F^2,M^2)$ a weak solution with the same initial data, additionally satisfying Prodi-Serrin conditions $|v^2|, |F^2|, |\nn M^2| \in L^r(0,T;L^{s}(\om))$, where $\frac{2}{r}+\frac{3}{s} = 1$ for some $s \in (3, \infty)$. Then $(v^1,F^1,M^1) = (v^2,F^2,M^2)$ a.e.\ in $Q_T$.
\end{thm}
\begin{rem}
Like in the classical Navier-Stokes equations, the borderline case $(r,s)=(2,\infty)$ is admissible as well and it can be proved by means of our proof to Theorem~\ref{unik}, mutatis mutandis. We avoid inclusion of this case to the proof just for aesthetic reasons; one would have to constantly distinguish between the cases $s<\infty$ and $s=\infty$, when considering fractions involving $s$. 

We conjecture that Theorem~\ref{unik} holds as well for $(r,s)=(\infty,3)$ but considering a comparatively more involved proof of this situation in Navier-Stokes equations, we do not aspire to prove it here. See~\cite{ribaud} for a further extension of the Prodi-Serrin conditions.
\end{rem}

\section{Uniqueness in 2D} \label{sec:uni2d}

\begin{proof}[Proof of Theorem \ref{thm1}] Let us have weak solution $(v^1,F^1,M^1)$, $(v^2,F^2,M^2)$ equipped with the same initial conditions. Denoting the corresponding systems with corresponding superscripts, we test equation~\eqref{prob1}$^1 -$~\eqref{prob1}$^2$ with $v\ddf v_1 -v_2$ over $Q_t \ddf (0,t)\times \om$, $t\in(0,T)$. The pressure term disappears due to the divergence-free condition on $v$ and we obtain
\begin{align} \label{opeth}
\frac12 \| v(t) \|_2^2 + \nu \int_{Q_t} | \nn v |^2 + I_0 + I_1 + I_2 = 0,
\end{align}
where 
\begin{align*}
I_0 & \ddf \int_{Q_t} ((v^1 \cdot \nabla)v^1-(v^2 \cdot \nabla)v^2) \cdot v = \int_{Q_t} ((v \cdot \nabla)v^1) \cdot v \\
I_1 & \ddf \int_{Q_t} \dir (\nn^\text{T} \! M^1 \, \nn M^1  - \nn^\text{T} \! M^2 \, \nn M^2  ) \cdot v, \\
I_2 & \ddf - \int_{Q_t} \dir (F^1(F^1)^\text{T} - F^2(F^2)^\text{T})\cdot v.
\end{align*}
The terms $I_1$, $I_2$ pollute~\eqref{opeth} and we get rid of them using equations \eqref{prob3}$^i$ and \eqref{prob4}$^i$. Let us test equation~\eqref{prob3}$^1 -$ \eqref{prob3}$^2$ with $F \ddf F^1 - F^2$ over $Q_t$:
\begin{align}
\frac12 \| F(t) \|_2^2 + \kappa \int_{Q_t} | \nn F |^2 + I_3 + I_4 = 0,
\end{align}
where 
\begin{align*}
I_3 & \ddf \int_{Q_t} ((v^1 \cdot \nabla)F^1-(v^2 \cdot \nabla)F^2) :F =  \int_{Q_t} ((v \cdot \nabla)F^2):F^1= \int_{Q_t} ((v \cdot \nabla)F^1):F, \\
I_4 & \ddf - \int_{Q_t} (\nn v^1 F^1-\nn v^2 F^2) :F.
\end{align*}
Similarly~\eqref{prob4}$^1 -$ \eqref{prob4}$^2$ with $M \ddf (M^1 - M^2)$:
\begin{align} 
\frac12 \| M(t) \|_2^2 + \int_{Q_t} | \nn M |^2 + I_5 + I_6 = 0,
\end{align}
where 
\begin{align*}
I_5 & \ddf  \int_{Q_t} ((v^1 \cdot \nabla)M^1-(v^2 \cdot \nabla)M^2) \cdot M =  \int_{Q_t} ((v \cdot \nabla)M^2)\cdot M^1= \int_{Q_t} ((v \cdot \nabla)M^1)\cdot M, \\
I_6 & \ddf \frac{1}{\mu^2} \int_{Q_t} ((|M^1|^2 - 1)M^1-(|M^2|^2 - 1)M^2) \cdot M.
\end{align*}
Lastly, test again ~\eqref{prob4}$^1 -$ \eqref{prob4}$^2$ with $-\Delta M \ddf -\Delta (M^1 - M^2)$:
\begin{align} \label{opet}
\frac12 \| \nn M(t) \|_2^2 + \int_{Q_t} | \Delta M |^2 + I_7 + I_8 = 0,
\end{align}
where 
\begin{align*}
I_7 & \ddf  - \int_{Q_t} ((v^1 \cdot \nabla)M^1-(v^2 \cdot \nabla)M^2)\cdot \Delta M, \\
I_8 & \ddf - \frac{1}{\mu^2} \int_{Q_t} ((|M^1|^2 - 1)M^1-(|M^2|^2 - 1)M^2) \cdot \Delta M.
\end{align*}
Now we sum up \eqref{opeth}--\eqref{opet}:
\begin{multline} \label{babel} 
\underbrace{\frac12 \big( \| v(t) \|_2^2 +\| F(t) \|_2^2 +\| M(t) \|_2^2 +\| \nn M(t) \|_2^2 \big)}_{\dfe f(t)} \\[-15pt] + \overbrace{\int_{Q_t} \big( \nu | \nn v |^2 + \kappa | \nn F |^2+ | \nn M |^2 + | \Delta M |^2 \big)}^{\dfe g(t)} + \sum_{i=0}^8 I_i = 0,
\end{multline}
where the sum $\sum_{i=0}^8 I_i$ has to be taken care of. There is no problem with $I_0$, $I_3$ and $I_5$ as these terms are each quadratic in the solution difference and we can obtain $|I_{0,3,5}| \le \int_0^t h(s) f(s) \, ds + \e g(t)$ for some non-negative $h \in L^1(0,T)$ and $0 < \e \ll 1$: Since $d=2$, we may avail ourselves of the interpolation
\begin{align*} 
\| u \|_4^2 \le C \| u \|_2 \| \nn u \|_2 \quad \text{for any $u \in W_0^{1,2}(\om)$ with $C=C(\om)$,}
\end{align*}
and when combined with Young's inequality, we get
\begin{align} \nonumber
|I_0| & \le \int_{Q_t} |v|^2 |\nabla v^1| \le \int_0^t \| v(s) \|_4^2 \| \nn v^1(s) \|_2 \le C \int_0^t \| v(s) \|_2 \| \nn v(s) \|_2 \| \nn v^1(s) \|_2 \\  \label{nula}
&\le C \int_0^t \| v(s) \|_2^2 \| \nn v^1(s) \|^2_2 + \int_{Q_t} \frac{\nu}{8} | \nn v |^2, \\ \nonumber
|I_3| & \le \int_{Q_t} |v| |F| |\nabla F^1| \le \int_0^t \| v(s) \|_4 \| F(s) \|_4 \| \nn F^1(s) \|_2 \\ \nonumber
& \le C \int_0^t \| v(s) \|_2^{1/2} \| \nn v(s) \|_2^{1/2} \| F(s) \|_2^{1/2} \| \nn F(s) \|_2^{1/2} \| \nn F^1(s) \|_2 \\ \label{tri}
&\le C \int_0^t \big( \| v(s) \|_2^2 + \| F(s) \|_2^2 \big) \| \nn F^1(s) \|^2_2 + \int_{Q_t}\frac18 \big(  \nu | \nn v |^2 + \kappa | \nn F |^2 \big), \\ \nonumber
|I_5| & \le \int_{Q_t} |v| |M| |\nabla M^1| \le \int_0^t \| v(s) \|_4 \| M(s) \|_4 \| \nn M^1(s) \|_2 \\ \nonumber
& \le C \int_0^t \| v(s) \|_2^{1/2} \| \nn v(s) \|_2^{1/2} \| M(s) \|_2^{1/2} \| \nn M(s) \|_2^{1/2} \| \nn M^1(s) \|_2 \\ \label{pet}
&\le C \int_0^t \big( \| v(s) \|_2^2 + \| M(s) \|_2^2 \big) \| \nn M^1(s) \|^2_2 + \int_{Q_t} \frac18 \big(  \nu | \nn v |^2 + | \nn M |^2 \big). 
\end{align}
If we are able to achieve the same also with other $I_i$'s, we may in the end invoke Gronwall's lemma and hence procure the result. 

Let us now take $I_2+I_4$, i.e.\ the so-far uncontrolled terms related to the deformation gradients. Given that $$(F^i(F^i)^\text{T}):\nn v^j = (\nn v^j F^i):F^i, \quad i,j\in\{1,2\}, $$
we have
\begin{align*} 
I_2+I_4 &= - \int_{Q_t} \big( (F^1(F^1)^\text{T}):\nn v^2 + (F^2(F^2)^\text{T}):\nn v^1 - (\nn v^1 F^1):F^2  -(\nn v^2 F^2) :F^1 \big) \\ 
& = - \int_{Q_t} \big( (\nn v^2 F^1) :F^1 + (\nn v^1 F^2) :F^2 - (\nn v^1 F^1):F^2  -(\nn v^2 F^2) :F^1 \big) \\
& = \int_{Q_t} \big( (\nn v^1 F) :F^2- (\nn v^2 F) :F^1  \big) = \int_{Q_t} \big(  (\nn v F) :F^1 -(\nn v^1 F) :F \big),
\end{align*}
each of which is again quadratic in the solution difference and therefore manageable completely analogously as before, i.e.
\begin{align} \label{problm}
|I_2+I_4| & \le \int_{Q_t} |\nn v| |F| |F^1|  + \int_{Q_t} |F|^2 |\nn v^1|
\end{align}
and 
\begin{align} \nonumber
\int_{Q_t} |\nn v| |F| |F^1| & \le C\int_0^t \| \nn v(s) \|_2 \| F(s) \|_2^{1/2} \| \nn F(s) \|_2^{1/2} \| F^1(s) \|_2^{1/2} \| \nn F^1(s) \|_2^{1/2}  \\ \label{dva}
&\le   C \int_0^t \| F(s) \|_2^2 \| \nn F^1(s) \|_2^2 + \int_{Q_t} \frac18 \big(  \nu| \nn v |^2 + \kappa | \nn F |^2 \big), \\ \label{ctyri}
 \int_{Q_t} |F|^2 |\nn v^1| & \le C \int_0^t \| F(s) \|_2^2 \| \nn v^1(s) \|^2_2 + \int_{Q_t} \frac{\kappa}{8} | \nn F |^2. 
\end{align}
Next we handle $I_1+I_7$. First
$$\dir (\nn^\text{T} \! M^i \, \nn M^i  ) = \frac12 \nn |\nn M^i |^2 + \nn^\text{T} \! M^i \, \Delta M^i, \quad i=1,2,$$
so that
\begin{align*}
I_1 &= \int_{Q_t} (\nn^\text{T} \! M^1 \Delta M^1  - \nn^\text{T} \! M^2 \Delta M^2  ) \cdot v \\& = \int_{Q_t}  \big(((v^1 \cdot \nabla)M^1)\cdot \Delta M^1 + ((v^2 \cdot \nabla)M^2) \cdot \Delta M^2 -(\nn^\text{T} \! M^1 \Delta M^1) \cdot v^2  - (\nn^\text{T} \! M^2 \Delta M^2  ) \cdot v^1\big), \\
\intertext{hence}
I_1 + I_7 &= \int_{Q_t} \big( (\nn^\text{T} \! M^1 \Delta M^1  - \nn^\text{T} \! M^2 \Delta M^2  ) \cdot v - ((v^1 \cdot \nabla)M^1-(v^2 \cdot \nabla)M^2)\cdot \Delta M \big) \\
&= \int_{Q_t} \big( ((v^1 \cdot \nabla)M^1)\cdot \Delta M^2 +((v^2 \cdot \nabla)M^2)\cdot \Delta M^1  -(\nn^\text{T} \! M^1 \Delta M^1) \cdot v^2  - (\nn^\text{T} \! M^2 \Delta M^2  ) \cdot v^1\big) \\
&= \int_{Q_t} \big((\nn^\text{T} \! M^1 \Delta M^2) \cdot v^1 + (\nn^\text{T} \! M^2 \Delta M^1) \cdot v^2  -(\nn^\text{T} \! M^1 \Delta M^1) \cdot v^2  - (\nn^\text{T} \! M^2 \Delta M^2  ) \cdot v^1 \big) \\
&= \int_{Q_t} \big((\nn^\text{T} \! M \Delta M^2) \cdot v^1 - (\nn^\text{T} \! M \Delta M^1) \cdot v^2\big) = \int_{Q_t} \big((\nn^\text{T} \! M \Delta M^2) \cdot v - (\nn^\text{T} \! M \Delta M) \cdot v^2 \big).
\end{align*}
Both terms at the end are again quadratic in the solution difference and we may apply the same estimates as before: 
\begin{align} \label{shot}
|I_1+I_7| & \le \int_{Q_t} |\nn M| |v| |\Delta M^2|  + \int_{Q_t} |\nn M| |\Delta M| |v^2|.
\end{align}
Recalling the standard regularity result for the Laplace equation $$ \| \nn^2 u \|_2 \le C \big(\| u \|_2+\| \Delta u \|_2\big) \quad \text{for any $u \in W^{2,2}(\om)$ with $C=C(\om)$,} $$
the two terms in \eqref{shot} can be handled as follows:
\begin{align} 
\int_{Q_t} |\nn M| |v| |\Delta M^2| & \le   C \int_0^t \| \nn M (s) \|_2^{1/2} \| \nn^2 M (s) \|_2^{1/2} \| v(s) \|_2^{1/2} \| \nn v(s) \|_2^{1/2}\| \Delta M^2 (s) \|_2 \nonumber \\
& \le C \int_0^t \| \nn M (s) \|_2^{1/2} \big( \| M (s) \|_2^{1/2}+\| \Delta M (s) \|_2^{1/2} \big) \| v(s) \|_2^{1/2} \| \nn v(s) \|_2^{1/2}\| \Delta M^2 (s) \|_2 \nonumber \\
 & \le C \int_0^t \big( \| v (s) \|_2^2 + \| M (s) \|_2^2 + \| \nn M (s) \|_2^2  \big) \| \Delta M^2(s) \|_2^2 \nonumber \\
 & \qquad + \int_{Q_t}\frac18 \big( |\nn M|^2 + | \Delta M |^2 + \nu |\nn v|^2 \big), \label{jedna} \\ \nonumber
 \int_{Q_t} |\nn M| |\Delta M| |v^2| & \le C \int_0^t \| \nn M (s) \|_2^{1/2} \| \nn^2 M (s) \|_2^{1/2} \| v^2(s) \|_2^{1/2} \| \nn v^2(s) \|_2^{1/2}\| \Delta M (s) \|_2 \\ \nonumber
 & \le C \int_0^t \| \nn M (s) \|_2 \big( \| M (s) \|_2 + \| \Delta M (s) \|_2 \big) \| v^2(s) \|_2 \| \nn v^2(s) \|_2 + \int_{Q_t}\frac{1}{16} | \Delta M |^2 \\
 & \le C \int_0^t \big( \| M (s) \|_2^2 + \| \nn M (s) \|_2^2 \big) \| v^2(s) \|_2 \| \nn v^2(s) \|_2 + \int_{Q_t}\frac{1}{8} | \Delta M |^2 \label{sedm} . 
\end{align}
Thus what remains is to deal with $I_6$ and $I_8$, for which we first take into account
\begin{gather*}
(|M^1|^2 M^1-|M^2|^2 M^2) \cdot (M^1 - M^2) \ge 0, \\
\left| |M^1|^2 M^1-|M^2|^2 M^2 \right| \le \frac32 |M| (|M^1|^2 +|M^2|^2),
\end{gather*}
so that
\begin{align} \label{sest}
I_6 & = - \frac{1}{\mu^2} \int_{Q_t} | M |^2 + \frac{1}{\mu^2} \int_{Q_t} (|M^1|^2 M^1-|M^2|^2 M^2) \cdot M \ge - \frac{1}{\mu^2} \int_{Q_t} | M |^2, \\
I_8 & = - \frac{1}{\mu^2} \int_{Q_t} | \nn M |^2 - \frac{1}{\mu^2} \int_{Q_t} (|M^1|^2 M^1-|M^2|^2 M^2) \cdot \Delta M \nonumber \\
& \ge - \frac{1}{\mu^2} \int_{Q_t} | \nn M |^2 - \frac{3}{2 \mu^2} \int_{Q_t} (|M^1|^2 +|M^2|^2) |M| |\Delta M| \nonumber \\ \label{osm}
& \ge - C \int_0^T  \big( \| M (s) \|_2^2 + \| \nn M (s) \|_2^2 \big) \big( 1+ \| M^1 (s) \|_8^4 + \| M^2 (s) \|_8^4 \big) - \int_{Q_t}\frac{1}{8} | \Delta M |^2.
\end{align}
Because $M^{1,2} \in L^4(0,T;L^8(\om))$, we can apply estimates \eqref{nula}--\eqref{osm} to equation \eqref{babel}, to finally obtain
\begin{align*}
f(t) \le \int_0^t h(s)f(s)\, ds
\end{align*}
for some non-negative, integrable function $h$. Gronwall's inequality now yields the claim.
\end{proof}

\section{Uniqueness in 3D} \label{sec:uni3d}

In this section we prove Theorem \ref{unik}. We begin with two auxiliary lemmas.
\begin{lem} \label{unreal}
Let $(v,F,M)$ be a weak solution in the sense of Definition \ref{defwsol}, enjoying additionally\footnote{For a possible way how to weaken these assumptions, see \cite{MR2665030}.} 
$$v \in L^4(0,T;L^4(\om;\R^{d})), \quad F \in L^4(0,T;L^4(\om;\R^{d\times d})), \quad \nn M \in L^4(0,T; L^4(\om;\R^{3\times d})). $$
Then $$v_t \in L^2(0,T;(W^{1,2}_{0,\dir}(\om))^*), \quad F_t \in (L^2(0,T;(W^{1,2}_0(\om)) \cap L^4(Q_T))^*, \quad M_t \in L^2(0,T; L^2(\om)).$$
In particular, $(v,F,M)$ satisfies \eqref{enineq} even as equality. 
\end{lem}
\begin{proof}
The energy equality results from the fact that we are allowed to test~\eqref{prob1} by $v$, \eqref{prob3} by $F$ and \eqref{prob4} by $M-\Delta M$. The sum of the three integral identities then yields the equality.

Now for the time derivatives. For $\varphi \in L^2(0,T;W^{1,2}_{0,\dir}(\om))$ we estimate the time derivative $v_t$ by means of~\eqref{prob1} as follows:
\begin{align*}
\Bigl| \int_0^T \langle &v_t,  \varphi \rangle \Bigr| \\ 
& = \Bigl| \int_0^T \nu \ii \nn v : \nn \varphi \, dx - \ii (v \otimes v) : \nn \varphi \, dx - \ii (\nn^\text{T} \! M \, \nn M ): \nn \varphi \, dx + \ii (FF^\text{T}) : \nn \varphi \, dx  \, dt \Bigr|  \\ 
& \le \int_0^T \Bigl( \nu \| \nn v \|_{L^2(\om)} + \| v \|_{L^4(\om)}^2 + \| \nn M \|_{L^4(\om)}^2 + \| F \|_{L^4(\om)}^2 \Bigr) \| \nn \varphi \|_{L^2(\om)} \, dt \\
& \le C \| \varphi \|_{L^2(0,T;W^{1,2}_{0,\dir}(\om))}. \phantom{\int_0^T}
\end{align*}
Similarly for $F_t$: Let $\varphi \in L^2(0,T;W^{1,2}_0(\om)) \cap L^4(Q_T)$. Then
\begin{align*}
\Bigl| \int_0^T \langle F_t, \varphi \rangle \Bigr| & = \Bigl| \int_0^T \kappa \ii \nn F : \nn \varphi \, dx + \ii (v \cdot \nabla)F : \varphi \, dx - \ii \nn v F : \varphi \, dx  \, dt \Bigr|  \\ 
& \le \int_0^T \kappa \| \nn F \|_{L^2(\om)} \| \nn \varphi \|_{L^2(\om)}  + \bigl( \| v \|_{L^4(\om)}\| \nn F \|_{L^2(\om)} + \| \nn v \|_{L^2(\om)}\| F \|_{L^4(\om)}\bigr) \| \varphi \|_{L^4(\om)} \, dt \\
& \le C \bigl(\| \varphi \|_{L^2(0,T;W^{1,2}_{0,\dir}(\om))}+ \| \varphi \|_{L^4(Q_T)} \bigr). \phantom{\int_0^T}
\end{align*}
Finally, since $M\in L^8(Q_T)$, the time derivative $M_t$ is estimated from~\eqref{prob4} simply as
\begin{align*}
\| M_t \|_{L^2(Q_T)} &\le \| (v \cdot \nabla)M \|_{L^2(Q_T)} + \| \Delta M \|_{L^2(Q_T)} + \frac{1}{\mu^2} \|(|M|^2 - 1)M \|_{L^2(Q_T)} \\
& \le \| v \|_{L^4(Q_T)} \| \nn M \|_{L^4(Q_T)} + \| \Delta M \|_{L^2(Q_T)} +  \frac{1}{\mu^2} \| M \|_{L^6(Q_T)}^3 + \frac{1}{\mu^2} \| M \|_{L^2(Q_T)}  \le C. \qedhere
\end{align*}
\end{proof}

Denoting the corresponding systems~\eqref{prob1}--\eqref{prob4} for $(v^i,F^i,M^i)$ as~\eqref{prob1}$^i$--\eqref{prob4}$^i$, $i=1,2$, another auxiliary lemma for the proof of Theorem~\ref{unik} is the following:
\begin{lem} \label{devin}
Under assumptions of Theorem \ref{unik}, we are allowed to test~\eqref{prob1}$^1$ with $v^2$, \eqref{prob3}$^1$ with $F^2$ and~\eqref{prob4}$^1$ with $M^2-\Delta M^2$.
\end{lem}
\begin{proof}
We proceed like in Lemma \ref{unreal}. Let us show that $v_t^1 \in (L^2(0,T;W^{1,2}_{0,\dir}(\om)) \cap L^r(0,T;L^{s}(\om)))^*$ first:
\begin{multline*}
\Bigl| \int_0^T \langle v_t^1, \varphi \rangle \Bigr| = \Bigl| \int_0^T \nu \ii \nn v^1 : \nn \varphi \, dx \\ - \ii (v^1 \cdot \nabla) v^1 \cdot \varphi \, dx + \ii \dir (\nn^\text{T} \! M^1 \, \nn M^1 ) \cdot \varphi \, dx - \ii \dir (F^1(F^1)^\text{T}) \cdot \varphi \, dx  \, dt \Bigr| 
\end{multline*}
It evidently suffices to estimate the terms on the second line.
\begin{align*}
\Bigl| \int_0^T \ii (v^1 \cdot \nabla) v^1 \cdot \varphi \, dx \, dt \Bigr| & \le  \int_0^T \|v^1\|_{L^{\frac{2s}{s-2}}(\om)} \|\nn v^1\|_{L^2(\om)} \|\varphi\|_{L^{s}(\om)} \, dx \, dt \\
& \le  C \int_0^T \|v^1\|_{L^2(\om)}^{1- \frac{3}{s}} \|\nn v^1\|_{L^2(\om)}^{1+\frac{3}{s}} \|\varphi\|_{L^{s}(\om)} \, dx \, dt
\\
& \le  C \, \|v^1\|_{L^{\infty}(0,T;L^2(\om))}^{1- \frac{3}{s}} \|\nn v^1\|_{L^2(Q_T)}^{1+\frac{3}{s}} \|\varphi\|_{L^{\frac{2s}{s-3}}(0,T;L^{s}(\om))} \\
& \le C \, \|\varphi\|_{L^{r}(0,T;L^{s}(\om))}, \displaybreak[0] \\ 
\Bigl| \int_0^T \ii \dir (\nn^\text{T} \! M^1 \, \nn M^1 ) \cdot \varphi \, dx \, dt \Bigr| & \le  2 \int_0^T \|\nn M^1\|_{L^{\frac{2s}{s-2}}(\om)} \|\nn^2 M^1\|_{L^2(\om)} \|\varphi\|_{L^{s}(\om)} \, dx \, dt \\
& \le  C \int_0^T \|\nn M^1\|_{L^2(\om)}^{1- \frac{3}{s}} \|\nn^2 M^1\|_{L^2(\om)}^{1+\frac{3}{s}} \|\varphi\|_{L^{s}(\om)} \, dx \, dt
\\
& \le  C \, \|\nn M^1\|_{L^{\infty}(0,T;L^2(\om))}^{1- \frac{3}{s}} \|\nn^2 M^1\|_{L^2(Q_T)}^{1+\frac{3}{s}} \|\varphi\|_{L^{\frac{2s}{s-3}}(0,T;L^{s}(\om))} \\
& \le C \, \|\varphi\|_{L^{r}(0,T;L^{s}(\om))}, \displaybreak[0] \\
\Bigl| \int_0^T \ii \dir (F^1(F^1)^\text{T}) \cdot \varphi \, dx \, dt \Bigr| & \le  2 \int_0^T \|F^1\|_{L^{\frac{2s}{s-2}}(\om)} \|\nn F^1\|_{L^2(\om)} \|\varphi\|_{L^{s}(\om)} \, dx \, dt \\
& \le  C \int_0^T \| F^1\|_{L^2(\om)}^{1- \frac{3}{s}} \|\nn F^1\|_{L^2(\om)}^{1+\frac{3}{s}} \|\varphi\|_{L^{s}(\om)} \, dx \, dt
\\
& \le  C \, \|F^1\|_{L^{\infty}(0,T;L^2(\om))}^{1- \frac{3}{s}} \|\nn F^1\|_{L^2(Q_T)}^{1+\frac{3}{s}} \|\varphi\|_{L^{\frac{2s}{s-3}}(0,T;L^{s}(\om))} \\
& \le C \, \|\varphi\|_{L^{r}(0,T;L^{s}(\om))}.
\end{align*}
Similarly $F_t^1 \in (L^2(0,T;W^{1,2}_0(\om)) \cap L^{r}(0,T;L^{s}(\om)))^*$:
\begin{align*}
\Bigl| \int_0^T \langle F^1_t, \varphi \rangle \Bigr| & = \Bigl| \int_0^T \kappa \ii \nn F^1 : \nn \varphi \, dx + \ii (v^1 \cdot \nabla)F^1 : \varphi \, dx - \ii \nn v^1 F^1 : \varphi \, dx  \, dt \Bigr|. 
\end{align*}
It is again sufficient to estimate the second term only (the third one follows analogously given that $v^1$ and $F^1$ have the same regularity): 
\begin{align} \label{mngld}
\left.
\begin{aligned}
\Bigl| \int_0^T \ii (v^1 \cdot \nabla)F^1 : \varphi \, dx \, dt \Bigr| & \le  \int_0^T \|v^1\|_{L^{\frac{2s}{s-2}}(\om)} \|\nn F^1\|_{L^2(\om)} \|\varphi\|_{L^{s}(\om)} \, dx \, dt \\ 
& \le  C \int_0^T \|v^1\|_{L^2(\om)}^{1- \frac{3}{s}} \|\nn v^1\|_{L^2(\om)}^{\frac{3}{s}} \|\nn F^1\|_{L^2(\om)}\|\varphi\|_{L^{s}(\om)} \, dx \, dt
\\ 
& \le  C \, \|v^1\|_{L^{\infty}(0,T;L^2(\om))}^{1- \frac{3}{s}} \|\nn v^1\|_{L^2(Q_T)}^{\frac{3}{s}} \|\nn F^1\|_{L^2(Q_T)} \|\varphi\|_{L^{\frac{2s}{s-3}}(0,T;L^{s}(\om))} \\ 
& \le C \, \|\varphi\|_{L^r(0,T;L^{s}(\om))}. 
\end{aligned}
\right.
\end{align}
And to conclude, $\nn M_t^1 \in \bigl(\{\varphi \in L^2(0,T;W^{1,2}(\om) \mid \varphi\cdot \n = 0 \text{\ on\ }(0,T)\times \partial \om \} \cap L^r(0,T;L^s(\om))\bigr)^*$:
\begin{align*}
\Bigl| \int_0^T \langle \nn M^1_t, \varphi \rangle \Bigr| &= \Bigl| \int_0^T \langle M^1_t, \dir \varphi \rangle \Bigr| = \Bigl| \int_0^T \ii \Delta M^1 \cdot \dir \varphi \, dx + \ii (v^1 \cdot \nabla)M^1 \cdot \dir \varphi \, dx \\ & \qquad + \frac{1}{\mu^2} \ii (|M^1|^2-1) M^1 \cdot \dir \varphi \, dx  \, dt \Bigr|. 
\end{align*}
The first and the last summand on the left are estimated trivially (recall $M^1 \in L^8(Q_T)$). As for the convective term, we estimate 
\begin{align*}
\Bigl| \int_0^T \ii (v^1 \cdot \nabla)M^1 \cdot \dir \varphi \, dx \, dt \Bigr| &\le \int_0^T \ii |\nn v^1| |\nabla M^1| |\varphi| + |v^1| |\nabla^2 M^1| |\varphi| \, dx \, dt 
\end{align*}
Both terms can be handled exactly like \eqref{mngld} and it follows that~\eqref{prob4}$^1$ can be tested with $\Delta M^2$ and hence also with $M^2-\Delta M^2$.
\end{proof}
Now we can prove the result on uniqueness in 3D:

\begin{proof}[Proof of Theorem \ref{unik}] Let us call the shared initial data $(v_0,F_0,M_0)$ and denote $v\ddf v_1 -v_2$, $F\ddf F_1 -F_2$ and $M\ddf M_1 -M_2$. Since, by simple interpolation,  $(v^2,F^2,M^2)$ satisfy the assumptions of Lemma~\ref{unreal}, in conjunction with Lemma~\ref{devin}, we have the following tools in our disposal:
\begin{enumerate}
    \item the energy inequality~\eqref{enineq} for  $(v^1,F^1,M^1)$,
    \item the energy equality (i.e.~\eqref{enineq} turned an equality) for $(v^2,F^2,M^2)$, 
    \item testability of~\eqref{prob1}$^1$, \eqref{prob3}$^1$ and \eqref{prob4}$^1$ by $v^2$, $F^2$ and $(M^2-\Delta M^2)$, respectively,
    \item testability of~\eqref{prob1}$^2$, \eqref{prob3}$^2$ and \eqref{prob4}$^2$ by $v^1$, $F^1$ and $(M^1-\Delta M^1)$, respectively.
\end{enumerate}
Schematically, let us perform $\text{1.}+ \text{2.}-(2\cdot \text{3.} + 2 \cdot  \text{4.})$. Since
\begin{align*}
\| v(t) \|_{L^2(\om)}^2 &= \| v^1(t) \|_{L^2(\om)}^2 + \| v^2(t) \|_{L^2(\om)}^2 - 2 \| v_0 \|_{L^2(\om)}^2 - 2 \int_0^t \Bigl(\langle v^1_t, v^2 \rangle + \langle v^2_t, v^1 \rangle \Bigr),  \\ 
 \| F(t) \|_{L^2(\om)}^2 &= \| F^1(t) \|_{L^2(\om)}^2 + \| F^2(t) \|_{L^2(\om)}^2 - 2 \| F_0 \|_{L^2(\om)}^2 - 2 \int_0^t \Bigl( \langle F^1_t, F^2 \rangle + \langle F^2_t, v^1 \rangle \Bigr),  \\
\| M(t) \|_{L^2(\om)}^2 &=\| M^1(t) \|_{L^2(\om)}^2 + \| M^2(t) \|_{L^2(\om)}^2 - 2 \| M_0 \|_{L^2(\om)}^2 - 2 \int_{Q_t} \Bigl( M^1_t \cdot M^2   + M^2_t \cdot M^1 \Bigr),  \\
\| \nn M(t) \|_{L^2(\om)}^2 &=\| \nn M^1(t) \|_{L^2(\om)}^2 + \| \nn M^2(t) \|_{L^2(\om)}^2 - 2 \| \nn M_0 \|_{L^2(\om)}^2 \\ & \quad - 2 \int_0^t \Bigl(\langle M^1_t, -\Delta M^2 \rangle + \langle M^2_t, -\Delta M^1 \rangle \Bigr), 
\end{align*}
we obtain a relation of the form 
\begin{multline} \label{blabel} 
\underbrace{\frac12 \big( \| v(t) \|_2^2 +\| F(t) \|_2^2 +\| M(t) \|_2^2 +\| \nn M(t) \|_2^2 \big)}_{\dfe f(t)} \\[-15pt] + \int_{Q_t} \big( \nu | \nn v |^2 + \kappa | \nn F |^2+ | \nn M |^2 + | \Delta M |^2 \big) \le  \sum_{i=1}^9 I_i,
\end{multline}
where
\begin{align*}
I_1 & \ddf \int_{Q_t} (v^1 \cdot \nabla)v^1 \cdot v^2 + \int_{Q_t} (v^2 \cdot \nabla)v^2 \cdot v^1,  \\
I_2 & \ddf  \int_{Q_t} \dir (\nn^\text{T} \! M^1 \, \nn M^1) \cdot v^2  + \int_{Q_t} \dir (\nn^\text{T} \! M^2 \, \nn M^2  ) \cdot v^1 , \\
I_3 & \ddf - \int_{Q_t} \dir (F^1(F^1)^\text{T}) \cdot v^2 - \int_{Q_t} \dir (F^2(F^2)^\text{T})\cdot v^1 \\
I_4 & \ddf  \int_{Q_t} (v^1 \cdot \nabla)F^1 : F^2 + \int_{Q_t} (v^2 \cdot \nabla)F^2 : F^1, \\
I_5 & \ddf -\int_{Q_t} (\nn v^1 F^1):F^2 - \int_{Q_t} (\nn v^2 F^2) :F^1, \\
I_6 & \ddf  \int_{Q_t} (v^1 \cdot \nabla)M^1 \cdot M^2 + \int_{Q_t} (v^2 \cdot \nabla)M^2 \cdot M^1, \\
I_7 & \ddf  - \frac{1}{\mu^2} \int_{Q_t} ((|M^1|^2 - 1)M^1 -(|M^2|^2 - 1)M^2) \cdot M, \\
I_8 & \ddf  - \int_{Q_t} (v^1 \cdot \nabla)M^1 \cdot \Delta M^2 - \int_{Q_t} (v^2 \cdot \nabla)M^2 \cdot \Delta M^1, \\
I_9 & \ddf \frac{1}{\mu^2} \int_{Q_t} ((|M^1|^2 - 1)M^1 -(|M^2|^2 - 1)M^2) \cdot \Delta M.
\end{align*}
Like in $d=2$, we want to estimate the sum $\sum_{i=1}^9 I_i$ so that we end up with $f(t) \le \int_0^t h(s) f(s) \, ds$ for some non-negative $h \in L^1(0,T)$, whence Gronwall's lemma yields the result. Let us start with $I_1$, which can be rewritten, due to orthogonality of the convective term, as
\begin{align*}
I_1 = \int_{Q_t} ((v^2 \cdot \nabla)v^2-(v^1 \cdot \nabla)v^1) \cdot v = - \int_{Q_t} (v \cdot \nabla)v^2 \cdot v = \int_{Q_t} (v \cdot \nabla)v \cdot v^2 .
\end{align*}
Prodi-Serrin conditions on $v^2$ now allow us to estimate\footnote{Note that $r=2s/(s-3)$.}
\begin{align} \nonumber
|I_1| & \le \int_{Q_t} |v| |\nn v| |v^2| \le \int_0^t \| v \|_{L^{\frac{2s}{s-2}}(\om)} \| \nn v \|_{L^2(\om)} \| v^2 \|_{L^s(\om)} 
 \le C \int_0^t \| v \|_{L^2(\om)}^{1-\frac{3}{s}} \| \nn v \|_{L^2(\om)}^{1+ \frac{3}{s}}  \| v^2 \|_{L^s(\om)} \\ &\le C \int_0^t \| v \|_{L^2(\om)}^2 \| v^2 \|^r_{L^s(\om)} +  \int_{Q_t}  \frac{\nu}{8} | \nn v |^2. \label{diojedna}
\end{align}
Terms $I_4$ and $I_6$, that can be rewritten as
\begin{align*}
I_4 &= \int_{Q_t} ((v^2 \cdot \nabla)F^2-(v^1 \cdot \nabla)F^1) : F = - \int_{Q_t} (v \cdot \nabla)F^2 : F = \int_{Q_t} (v \cdot \nabla)F : F^2, \\
I_6 &= \int_{Q_t} (v \cdot \nabla)M \cdot M^2.
\end{align*}
are treated analogously to $I_1$ (though not optimal, $I_6$ can be estimated by the same means as $I_4$):
\begin{align}  \nonumber
|I_4| & \le \int_{Q_t} |v| |\nn F| |F^2| \le \int_0^t \| v \|_{L^{\frac{2s}{s-2}}(\om)} \| \nn F \|_{L^2(\om)} \| F^2 \|_{L^s(\om)} \\
& \le C \int_0^t \| v \|_{L^2(\om)}^{1-\frac{3}{s}} \| \nn v \|_{L^2(\om)}^{\frac{3}{s}} \| \nn F \|_{L^2(\om)}  \| F^2 \|_{L^s(\om)} \nonumber \\
& \le C \int_0^t \| v \|_{L^2(\om)}^2 \| F^2 \|^r_{L^s(\om)} + \int_{Q_t}  \frac18 \big(  \nu | \nn v |^2 + \kappa | \nn F |^2 \big), \label{dioctyri} \\ \label{diosest}
|I_6| & \le C \int_0^t \| v \|_{L^2(\om)}^2 \| M^2 \|^r_{L^s(\om)} + \int_{Q_t}  \frac18 \big(  \nu | \nn v |^2 + | \nn M |^2 \big). 
\end{align}
Next we investigate the remaining terms related to the deformation gradients, i.e.\  $I_3+I_5$. Since $(F^i(F^j)^\text{T}):\nn v^k = (\nn v^k F^j):F^i,$ for $i,j,k=1,2$, we rewrite the sum as
\begin{align*} 
I_3+I_5 &= \int_{Q_t} \big( (F^1(F^1)^\text{T}):\nn v^2 + (F^2(F^2)^\text{T}):\nn v^1 - (\nn v^1 F^1):F^2  -(\nn v^2 F^2) :F^1 \big) \\ 
& =  \int_{Q_t} \big( (\nn v^2 F^1) :F^1 + (\nn v^1 F^2) :F^2 - (\nn v^1 F^1):F^2  -(\nn v^2 F^2) :F^1 \big) \\
& = \int_{Q_t} \big( (\nn v^2 F) :F^1- (\nn v^1 F) :F^2  \big) = \int_{Q_t} \big(  (\nn v^2 F) :F -(\nn v F) :F^2 \big) \\
& = - \int_{Q_t} \big(  \dir (FF^\text{T}) \cdot v^2 + (\nn v F) :F^2 \big),
\end{align*}
so that
\begin{align} \nonumber 
|I_3+I_5| & \le 2 \int_{Q_t} |\nn F| |F| |v^2|  + \int_{Q_t} |\nn v| |F| |F^2|.
\end{align}
Hence, much like in \eqref{diojedna} and \eqref{dioctyri}, we estimate
\begin{align} \nonumber
 \int_{Q_t} |\nn F| |F| |v^2| & \le C \int_0^t \| F \|_{L^2(\om)}^2 \| v^2 \|^{\frac{2s}{s-3}}_{L^s(\om)} +  \int_{Q_t}  \frac{\kappa}{8} | \nn F |^2. \\
 \int_{Q_t} |\nn v| |F| |F^2| & \le C \int_0^t \| F \|_{L^2(\om)}^2 \| F^2 \|^{\frac{2s}{s-3}}_{L^s(\om)} + \int_{Q_t}  \frac18 \big(  \nu | \nn v |^2 + \kappa | \nn F |^2 \big) \nonumber
\end{align}
and finally 
\begin{align} \label{diotriapet} 
|I_3+I_5| & \le C \int_0^t \| F \|_{L^2(\om)}^2 \big( \| v^2 \|^r_{L^s(\om)}  + \| F^2 \|^r_{L^s(\om)} \big) + \int_{Q_t}  \frac38 \big(  \nu | \nn v |^2 + \kappa | \nn F |^2 \big).
\end{align}
Let us continue with $I_2+I_8$. Since
$$\dir (\nn^\text{T} \! M^i \, \nn M^i  ) = \frac12 \nn |\nn M^i |^2 + \nn^\text{T} \! M^i \, \Delta M^i, $$
$I_2$ can be expressed as
\begin{align*}
I_2 = \int_{Q_t} (\nn^\text{T} \! M^1 \Delta M^1) \cdot v^2  + \int_{Q_t} (\nn^\text{T} \! M^2 \Delta M^2  ) \cdot v^1 = \int_{Q_t}  (v^2 \cdot \nabla)M^1\cdot \Delta M^1 + \int_{Q_t}  (v^1 \cdot \nabla)M^2\cdot \Delta M^2. 
\end{align*}
Therefore
\begin{align*}
I_2 + I_8 &= \int_{Q_t} \big( (v^2 \cdot \nabla)M^1\cdot \Delta M^1 + (v^1 \cdot \nabla)M^2\cdot \Delta M^2 - (v^2 \cdot \nabla)M^2\cdot \Delta M^1 - (v^1 \cdot \nabla)M^1\cdot \Delta M^2 \big) \\
&= \int_{Q_t} \big( (v^2 \cdot \nabla)M\cdot \Delta M^1 - (v^1 \cdot \nabla)M\cdot \Delta M^2 \big) = \int_{Q_t} \big( (v^2 \cdot \nabla)M\cdot \Delta M + (\nn (v \cdot \nabla)M)\cdot \nn M^2 \big).
\end{align*}
Thus we can estimate
\begin{align*}
|I_2+I_8| & \le \int_{Q_t} |v^2| |\nn M| |\Delta M|  + \int_{Q_t} |v| |\nn^2 M| |\nn M^2| +  \int_{Q_t} |\nn v| |\nn M| |\nn M^2|.
\end{align*}
Recalling $ \| \nn^2 u \|_2 \le C \big(\| u \|_2+\| \Delta u \|_2\big)$ for any $u \in W^{2,2}(\om)$ with $C=C(\om)$, we may again use the same technique as in \eqref{diojedna} and \eqref{dioctyri}. Note that $r> 2$ and for any $x,y \ge 0$ we have $(x+y)^{3/s} \le x^{3/s} + y^{3/s}$:
\begin{align*} 
\int_{Q_t} |v^2| |\nn M| |\Delta M| & \le C \int_0^t \| v^2 \|_s \| \nn M \|_2^{1- \frac{3}{s}} \| \Delta M \|_2 \bigl( \| M \|_2^{\frac{3}{s}} + \| \Delta M \|_2^{\frac{3}{s}} \bigr) \\
& \le C \int_0^t \Bigl( \| v^2 \|_s^2 \| M \|_2^{\frac{6}{s}} \| \nn M \|_2^{2- \frac{6}{s}} + \| v^2 \|_s^r \| \nn M \|_2^2\Bigr)  + \int_{Q_t} \frac18 | \Delta M |^2 \\
& \le C \int_0^t \Bigl( 1+ \| v^2 \|_s^r \Bigr) \Bigl( \| M \|_2^2 + \| \nn M \|_2^2 \Bigr)  + \int_{Q_t} \frac18 | \Delta M |^2, \displaybreak[0]\\
\int_{Q_t} |v| |\nn^2 M| |\nn M^2| & 
\le C \int_0^t \| \nn M^2 \|_s  \| v \|_2^{1- \frac{3}{s}} \| \nn v \|_2^{\frac{3}{s}} \bigl( \| M \|_2 + \| \Delta M \|_2 \bigr) \\
& \le C \int_0^t \Bigl( 1+ \| \nn M^2 \|_s^r \Bigr) \Bigl( \| v \|_2^2 + \| M \|_2^2 \Bigr)  + \int_{Q_t} \frac18 \bigl( \nu |\nn v|^2 + | \Delta M |^2 \bigr), \displaybreak[0]\\
\int_{Q_t} |\nn v| |\nn M| |\nn M^2| & \le C \int_0^t \|\nn M^2 \|_s \| \nn M \|_2^{1- \frac{3}{s}} \| \nn v \|_2 \bigl( \| M \|_2^{\frac{3}{s}} + \| \Delta M \|_2^{\frac{3}{s}} \bigr) \\
& \le C \int_0^t \Bigl( \| \nn M^2 \|_s^2 \| M \|_2^{\frac{6}{s}} \| \nn M \|_2^{2- \frac{6}{s}} + \| \nn M^2 \|_s^r \| \nn M \|_2^2\Bigr)  + \int_{Q_t} \frac{\nu}{8} | \nn v |^2 \\
& \le C \int_0^t \Bigl( 1+ \| \nn M^2 \|_s^r \Bigr) \Bigl( \| M \|_2^2 + \| \nn M \|_2^2 \Bigr)  + \int_{Q_t} \frac{\nu}{8} | \nn v |^2.
\end{align*}
All in all,
\begin{align} \label{diodvaaosm} 
|I_2+I_8| & \le C \int_0^t \Bigl( 1+ \| v^2 \|_s^r + \| \nn M^2 \|_s^r \Bigr) \Bigl(  \| v \|_2^2 + \| M \|_2^2 + \| \nn M \|_2^2 \Bigr)  + \int_{Q_t} \frac14 \bigl( \nu |\nn v|^2 + | \Delta M |^2 \bigr).
\end{align}
The last two terms $I_7$ and $I_9$, are easily dealt with exactly like in the case $d=2$: First,
\begin{align*}
(|M^1|^2 M^1-|M^2|^2 M^2) \cdot (M^1 - M^2) & \ge 0, \\
\left| |M^1|^2 M^1-|M^2|^2 M^2 \right| &\le \frac32 |M| (|M^1|^2 +|M^2|^2),
\end{align*}
which implies
\begin{align} \label{diosedm}
I_7 & = \frac{1}{\mu^2} \int_{Q_t} | M |^2 - \frac{1}{\mu^2} \int_{Q_t} (|M^1|^2 M^1-|M^2|^2 M^2) \cdot M \le \frac{1}{\mu^2} \int_{Q_t} | M |^2, \\
I_9 & = \frac{1}{\mu^2} \int_{Q_t} | \nn M |^2 + \frac{1}{\mu^2} \int_{Q_t} (|M^1|^2 M^1-|M^2|^2 M^2) \cdot \Delta M \nonumber \\
& \le  \frac{1}{\mu^2} \int_{Q_t} | \nn M |^2 + \frac{3}{2 \mu^2} \int_{Q_t} (|M^1|^2 +|M^2|^2) |M| |\Delta M| \nonumber \\ \nonumber
& \le \frac{1}{\mu^2} \int_{Q_t} | \nn M |^2 + C \int_0^t  \big( \| M^1 \|_8^4 + \| M^2 \|_8^4 \big)\| M\|_4^2  + \int_{Q_t}\frac{1}{8} | \Delta M |^2\\ \label{diodevet}
& \le C \int_0^T  \big( 1+ \| M^1 \|_8^4 + \| M^2 \|_8^4 \big) \big( \| M \|_2^2 + \| \nn M \|_2^2 \big) + \int_{Q_t}\frac{1}{8} | \Delta M |^2.
\end{align}
Because $M^{1,2} \in L^4(0,T;L^8(\om))$, we can apply estimates \eqref{diojedna}--\eqref{diodevet} to relation \eqref{blabel}, to finally obtain
\begin{align*}
f(t) \le \int_0^t h(s)f(s)\, ds
\end{align*}
for some non-negative, integrable function $h$. Gronwall's inequality now yields the claim.
\end{proof}

% ----%
\appendix

\section{Derivation of the model}  \label{sec:deriv}

This part follows \cite{BenesovaForster_etal2016,ForsterDiss}, which in turn were inspired by energetical variational approaches in modeling complex fluids, as for instance viscoelastic flows \cite{LiuWalkington2001} or nematic liquid crystal flows \cite{SunLiu2009}.

System \eqref{prob1}--\eqref{prob4} is completely phrased in the current configuration using Eulerian coordinates. In the derivation, however, the reference configuration of the magnetoelastic material is required. We briefly review the basic concepts: Let $\widetilde\Omega\subset \R^d$, $d=2,3$ denote the reference configuration. The reference and the current configurations are identified at $t=0$. The elements of $\widetilde\Omega$ are denoted by capital letters $X$ and are referred to as Lagrangian coordinates. The reference configuration is mapped to the current configuration at time $t$ by the deformation (or flow) map
$$x: [0,T]\times \widetilde\Omega\to\Omega,\; (t,X)\mapsto x(t,X), \quad \mbox{ where } x(0,X)=X.$$
The velocity $v:[0,T]\times\Omega \to\R^d$ satisfies 
$$v(t,x(t,X))=\frac{\partial}{\partial t} x(t,X) \quad \mbox{ for all } t>0.$$
 The  deformation gradient in Lagrangian coordinates is defined as 
 $$\widetilde F: [0,T]\times \widetilde\Omega \to \R^{d\times d} \quad \mbox{ with } \widetilde F(t,X):=\frac{\partial x}{\partial X}(t,X)$$
 and it is related to the deformation gradient $F$ in Eulerian coordinates introduced earlier
 through $F(t,x(t,X)) = \widetilde F(t,X)$. By an application of the chain rule, we obtain 
 \begin{align*} 
 F_t + (v\cdot\nabla) F-\nabla v F =0 \quad \mbox{in } [0,T]\times\Omega.
\end{align*}
Following \cite{LiuWalkington2001}, this equation allows to determine the deformation gradient in Eulerian coordinates from the velocity.
Note that our equation \eqref{prob3} enjoys an additional term on the right-hand side, namely $\kappa \Delta F$ for some $\kappa >0$. This is a regularization of the equation that is essential in the existence proof; see \cite{LinLiuZhang2005} and \cite[Remark~5]{BenesovaForster_etal2016} for a corresponding discussion.

From now on, let us assume the mass density $\rho$ to be constantly one. Furthermore, we work in the framework of incompressible materials, hence $\det F =\det \widetilde F = 1$ and
\begin{align} \label{incom}
\dir v=0 \quad \mbox{in } [0,T]\times\Omega,
\end{align}
which is \eqref{prob2}. To derive the momentum equation~\eqref{prob1} and the evolution equation for the magnetization~\eqref{prob4}, we follow the energetic variational approach and start from an action functional $\mathcal{A}(x)$, which is an integral of the difference of the kinetic energy and the Helmholtz free energy $\psi$ of the magnetoelastic material. 
We work with the Helmholtz free energy
\begin{align*}
 \psi(F,M) = & \frac12 \int_{\Omega} |\nabla M|^2~dx 
 - \int_\Omega M\cdot H_{\text{ext}}~dx 
   + \frac{1}{4\mu^2} \int_\Omega  (|M|^2-1)^2~dx + \int_\Omega W(F)~dx,
 \end{align*}
which is a special version of the micromagnetic energy neglecting the anisotropy energy and the stray field contribution; see again \cite{BenesovaForster_etal2016,ForsterDiss} for a discussion. Besides, the Helmholtz free energy contains a length penalization for the magnetization and the elastic energy of the system.  Note that we set various physical constants to one and that the elastic energy density $W:\R^{d\times d} \to \R$ is phrased in Eulerian coordinates. Here we choose $W(F)= \tfrac12 |F|^2$ for simplicity. While the existence proof holds for more general convex $W$, the uniqueness proofs of this article are restricted to this special choice of the elastic energy density; see e.g.\ derivation of~\eqref{problm}.

The transport of the magnetization is based on the assumption that the magnetization $\widetilde M: [0,T] \times \widetilde\Omega \to \R^3$ in the reference configuration and $M$ in the current configuration are related through $\widetilde M(t,X) = M(t,x(t,X))$. Hence $\frac{d}{dt} \widetilde M(t,X) = M_t + (v\cdot \nabla) M$ and the gradient flow equation for $M$ reads
$$ M_t + (v \cdot \nabla)M  = - \frac{\delta\psi }{\delta M} (F,M) = \Delta M + H_{\text{ext}} - \frac{1}{\mu^2}(|M|^2 - 1)M,$$
where $\frac{\delta\psi }{\delta M}$ stands for the variational derivative of $\psi$ with respect to $M$. Note that the above equation is actually our equation~\eqref{prob4}.

In order to derive equation \eqref{prob1}, we transform the action functional to Lagrangian coordinates: 
\begin{align*}
 \mathcal{A}(x) =& \int_0^T \int_{\widetilde \Omega} \frac12 |x_t(t,X)|^2 - \frac12 | \nabla_X \widetilde M(t,X) \widetilde F^{-1}|^2 + \widetilde M(t,X) \cdot H_\text{ext}(t,x(t,X)) \\ & \qquad\qquad -\frac1{4\mu^2}(|\widetilde M(t,X)|^2 -1)^2- |\widetilde F|^2~dX~dt.
\end{align*} 
By Hamilton's principle, the first variation of $\mathcal{A}$ with respect to the flow map $x$ equals zero. Therefore we consider a one-parameter volume-preserving diffeomorphism $x^\varepsilon$ with $x^0=x$ and $\frac{d x^\varepsilon}{d\varepsilon} \big|_{\varepsilon=0} = y$ and $\operatorname{det} \widetilde F^\varepsilon = \operatorname{det} \frac{\partial x^\varepsilon}{\partial X} = 1$ for all $\varepsilon>0$. The latter condition implies $\dir y =0$ since
$$ 0= \frac{d}{d\varepsilon}\operatorname{det} \widetilde F^\varepsilon  \bigg|_{\varepsilon=0} = \operatorname{det} \widetilde F \operatorname{tr}\bigg[\bigg(\frac{d}{d\varepsilon}\frac{\partial x^\varepsilon}{\partial X} \bigg|_{\varepsilon=0} \widetilde F^{-1}\bigg)\bigg] = \operatorname{det} \widetilde F \dir y = \dir y.$$
After some calculations including repeated integration by parts and a transformation to Eulerian coordinates, we obtain
\begin{align*}
 0 =& \frac{d}{d\varepsilon} \mathcal{A}(x^\varepsilon) \bigg|_{\varepsilon=0} =  \int_0^T \int_\Omega \bigg( - (v_t + (v\cdot \nabla)v)  - \dir (\nn^\text{T} \! M\, \nn M) + \nabla^\text{T} H_{\text{ext}} M + \dir(F F^T) \bigg) \cdot y ~dx~dt
\end{align*}
for an arbitrary solenoidal function $y$. The incompressibility condition can be dealt with as a Lagrange multiplier $p_1$, which yields
\begin{align*} 
v_t + (v \cdot \nabla) v + \nn p_1 + \dir (\nn^\text{T} \! M \, \nn M ) - \dir (FF^\text{T}) & = + 
\nabla^\text{T} H_{\text{ext}} M.
\end{align*}

Finally, adding dissipation $\int_\Omega \nu | \nabla v|^2~dx$ to the above model and considering the first variation with respect to the divergence-free velocity vector yields $-\nu \Delta v = \nabla p_2$  for some $p_2$. Defining the total pressure $p:= p_1-p_2$, we obtain~\eqref{prob1}. 

\section{Skeleton proof of the existence theorem} \label{sec:exist}

The idea of Theorem~\ref{thm:exist}, which is based on~\cite{ForsterDiss} and~\cite{LinLiu1995}, uses the Galerkin approach and a fixed-point argument. 
A similar proof is presented in~\cite{BenesovaForster_etal2016}, yet in a more involved version that is attributable to  an increased complexity of the equation for the magnetization. This is why we only sketch the proof here.

 \paragraph{Step 1} (Existence of approximate solution): Velocity field $v$ in the balance of linear momentum~\eqref{prob1} is spatially discretized by means of the eigenfunctions of the Stokes operator on $\om$ that form a sequence $(\xi_m)_{m\in\N} \subset W^{2,\infty}(\om)$ due to the assumed regularity of $\partial \om$.  As such, with $m\in\N$ fixed we look for approximate solutions $v_m$ of the form 
$$v_m(t,x) = \sum_{i=1}^m g_m^i(t)\xi_i(x) \mbox{ in }[0,T]\times \om,$$ 
where the unknown coefficients $g_m^i$ are absolutely continuous solutions of the ordinary differential equation
obtained from \eqref{prob1}, i.e.\ $g_m^i$ solves
\begin{align} \label{bindestrich}
\frac{d}{dt} g_m^i(t) = -\nu\lambda_i g_m^i(t) + \sum_{j,k=1}^m g_m^j(t) g_m^k(t) A_{jk}^i + D_m^i(t), \qquad i=1,\ldots,m,
\end{align}
where
\begin{gather*}
A_{jk}^i = -\int_\Omega (\xi_j(x)\cdot\nabla)\xi_k(x) \cdot \xi_i(x)~dx, \\[0.5em]
 D_m^i(t) = \int_\Omega \big( \nabla^\text{T} M_m\nabla M_m - F_mF_m^\text{T} \big) : \nabla\xi_i~dx.
\end{gather*}
The approximate solutions $F_m$ and $M_m$ are at this point considered to be known:   
For $t_0 \in (0,T]$ and a suitably large $L$, we define
$$V_m(t_0) \ddf \Bigl\{ \sum_{i=1}^m g^i(t) \xi_i(x): g^i \in C([0,T]), \max_{t\in[0,t_0]} \left| g^i(t) \right|\leq L, \ g^i(0) = \int_\Omega v_0(x) \cdot \xi_i(x)~dx,\ 1\le i\le m \Bigr\},$$
a closed and convex subset of $C([0,T]; H_m)$, where $H_m \ddf \spn \{ \xi_1, \ldots, \xi_m \}$. While the momentum equation~\eqref{prob1} is discretized, the equations for the deformation gradient \eqref{prob3} and for the magnetization \eqref{prob4} are left unchanged. By using another Galerkin scheme we can show that \eqref{prob3} and \eqref{prob4} have unique weak solutions $(F_m, M_m)$ for any given $v_m\in V_m(t_0)$ on a time interval $(0,t_1)$ for some $t_1 \leq t_0$. In addition, these solutions are sufficiently regular that we may, in turn, solve equation~\eqref{bindestrich}, on a possibly yet smaller time interval $(0,t^\ast)$. This yields a solution $\widetilde v$ which belongs to $V_m(t^\ast) $. We then show that the range of 
$$\mathcal{L}: V_m(t^\ast) \to V_m(t^\ast), \quad v\mapsto \widetilde v$$
 is precompact in  $C([0,t^\ast];H_m)$ and that $\mathcal{L}$ is continuous on $V(t^\ast)$ in the topology of $C([0,t^\ast];H_m)$. Hence Schauder's fixed point theorem yields existence of $ v_m \in V_m(t^\ast)$ such that $\mathcal{L}(v_m)=v_m$.
 This $v_m$ together with the corresponding $F_m$ and $M_m$ then form the approximate weak solution to the system \eqref{prob1}--\eqref{prob4} on the time interval $[0,t^\ast]$ for a fixed $m\in\N$.
 
\paragraph{Step 2} (Uniform estimates): We test the discrete weak formulation of \eqref{prob1} by
$v_m$, the discrete weak formulation of \eqref{prob3} by $F_m$, and the discrete weak formulation of \eqref{prob4} by $-\Delta M_m + \frac{1}{\mu^2} (| M_m |^2-1)M_m$. This yields the energy inequality
\begin{align*}
& \sup_{t \in (0, t^\ast)} \left( \int_\Omega |v_m(t)|^2 + |F_m(t)|^2 + |\nabla  M_m(t)|^2 + \frac{1}{2\mu^2}(|M_m(t)|^2-1)^2~dx \right)  \\
\nonumber &\qquad \qquad + 2 \int_0^{t^\ast}\int_\Omega \nu |\nabla  v_m|^2 + \kappa |\nabla F_m|^2 + \left| \Delta M_m - \frac{1}{\mu^2} (|M_m |^2-1) M_m \right|^2~dx~ds \\
& \qquad \qquad \qquad \qquad\leq  \int_\Omega | v_0|^2 + |F_0|^2 + |\nabla M_0|^2 + \frac{1}{2\mu^2}(| M_0|^2-1)^2~dx.
\end{align*}
We may use this inequality as a stepping stone to justify extendibility of the approximate solutions onto the entire time interval $(0,T)$ without having to change the above energy inequality. By this means we obtain a bound on the solutions in reflexive spaces independently of $m\in\N$.

\paragraph{Step 3} (Uniform estimates of the time derivatives): In order to achieve strong convergence of $(v_m, F_m, M_m)$ in the corresponding Bochner spaces needed in order to pass to the limit in the nonlinearities of the system, we have yet to prove uniform bounds on time derivatives. These bounds are obtained directly from the approximated system in the standard way as for the Navier-Stokes equations, cf., e.g., \cite{Robinson-etal,temam1977navier}.

\paragraph{Step 4} (Convergence of the discrete system): By the Aubin-Lions lemma and the estimates we have derived up to this point, we may take a subsequence of approximate solutions such that 
\begin{eqnarray*}
& v_m\to v &\quad\text{in } L^2(0,T;L^2(\Omega;\R^d)), \\
& \nabla v_m \rightharpoonup \nabla v &\quad\text{in } L^2(0,T;L^2(\Omega;\R^{d\times d})), \\
& F_m\to F &\quad\text{in } L^2(0,T;L^2(\Omega;\R^{d\times d})), \\
& \nabla F_m \rightharpoonup \nabla F &\quad\text{in } L^2(0,T;L^2(\Omega;\R^{d\times d\times d})),\\
& \nabla M_m\to \nabla M &\quad\text{in } L^2(0,T;L^2(\Omega;\R^{3\times d})),\\
& \Delta M_m\rightharpoonup \Delta M &\quad\text{in } L^2(0,T;L^2(\Omega;\R^3)).
\end{eqnarray*}
A passage to the limit in $m$ is now possible. We find a weak solution to system \eqref{prob1}--\eqref{prob4} in the sense of Definition~\ref{defwsol}, which also can be shown to attain the initial data as required by~\eqref{icond} and inequality~\eqref{enineq}. This then finishes the proof of Theorem~\ref{thm:exist}.

%----%

\def\ocirc#1{\ifmmode\setbox0=\hbox{$#1$}\dimen0=\ht0 \advance\dimen0
  by1pt\rlap{\hbox to\wd0{\hss\raise\dimen0
  \hbox{\hskip.2em$\scriptscriptstyle\circ$}\hss}}#1\else {\accent"17 #1}\fi}

%\bibliographystyle{plain}
%\bibliography{zabian}
\end{document}